\documentclass[11pt,reqno]{amsart}
\usepackage{amsmath,amsxtra,latexsym,amsthm,amssymb,amscd,pb-diagram}
\usepackage{mathrsfs,mathabx,amsfonts}
\usepackage[colorlinks=true,linkcolor=red,citecolor=blue]{hyperref}
\usepackage[margin=1in]{geometry}
\usepackage{bm}
\usepackage{color}
\usepackage{makecell,multirow,diagbox}
\usepackage{mathtools}
\usepackage[displaymath, mathlines]{lineno}
\usepackage{verbatim} 
\usepackage{graphicx}
\usepackage{epsfig}
\usepackage{array}
\usepackage{enumitem}

\usepackage[latin1]{inputenc}
\usepackage{tikz}
\usetikzlibrary{decorations.pathreplacing,calc}
\usepackage{rotating}
\tikzset{liltext/.style={font=\tiny}}
\usepackage{array}
\usepackage{float}

\newtheorem{definition}{Definition}[section]
\newtheorem{theorem}{Theorem}[section]

\newtheorem{lemma}{Lemma}[section]

\newtheorem{remark}{Remark}[section]
\newtheorem{example}{Example}[section]

\newcommand{\R}{\mathbb{R}}

\newcommand{\Z}{\mathbb{Z}}
\newcommand{\N}{\mathbb{N}}
\newcommand{\e}{\varepsilon}
\newcommand{\ity}{\infty}
\newcommand{\dps}{\displaystyle}
\newcommand{\f}{\displaystyle\frac}

\begin{document}
\title[Nonlinear damped wave equations on weighted graphs]{Blow up results and lifespan estimates for nonlinear damped wave equations on weighted graphs}
\thanks{\textit{Corresponding author:} Tuan Anh Dao}

\maketitle
\centerline{\scshape \textbf{Tuan Anh Dao}}
{\footnotesize
	\centerline{Faculty of Mathematics and Informatics, Hanoi University of Science and Technology}
	\centerline{No.1 Dai Co Viet road, Hanoi, Vietnam}
	\centerline{Email: anh.daotuan@hust.edu.vn}}
\medskip

\centerline{\scshape \textbf{Anh Tuan Duong}}
{\footnotesize
	\centerline{Faculty of Mathematics and Informatics, Hanoi University of Science and Technology}
	\centerline{No.1 Dai Co Viet road, Hanoi, Vietnam}
	\centerline{Email: tuan.duonganh@hust.edu.vn}}

\begin{abstract}
 In this article, we are interested in studying the Cauchy problems for nonlinear damped wave equations and their systems on a weighted graph. Our main purpose is two-fold, namely, under certain conditions for volume growth of a ball and the initial data we would like to not only prove nonexistence of global (in time) weak solutions but also indicate lifespan estimates for local (in time) weak solutions when a blow-up phenomenon in finite time occurs. Throughout the present paper, we will partially give a positive answer for the optimality of our results by an application to the $n$-dimensional integer lattice graph $\Z^n$ to recover the well-known results in the Euclidean setting.
\end{abstract}

\tableofcontents

\section{Introduction}
In this paper, let us consider the following Cauchy problems for nonlinear damped wave equations on a weighted graph $G=(V,E)$:
\begin{equation}
\begin{cases}
u_{tt}(t,x)- \Delta u(t,x)+ u_t(t,x)= |u(t,x)|^p, &(t,x) \in (0,T) \times V, \\
u(0,x)= \e u_0(x),\quad u_t(0,x)= \e u_1(x), &\quad x\in V,
\end{cases}
\label{Equation_Main}
\end{equation}
and
\begin{equation}
\begin{cases}
u_{tt}(t,x)- \Delta u(t,x)+ u_t(t,x)= |v(t,x)|^p, &(t,x) \in (0,T) \times V, \\
v_{tt}(t,x)- \Delta v(t,x)+ v_t(t,x)= |u(t,x)|^q, &(t,x) \in (0,T) \times V, \\
u(0,x)= \e u_0(x),\quad u_t(0,x)= \e u_1(x), &\quad x\in V, \\
v(0,x)= \e v_0(x),\quad v_t(0,x)= \e v_1(x), &\quad x\in V,
\end{cases}
\label{System_Main}
\end{equation}
where $p,q>0$ are called power exponents, $T>0$ and $V,E$ are sets of vertices, edges of a graph $G$, respectively. The operator $\Delta$ stands for the Laplacian on $V$. The functions $u,v:(0,T)\times V \rightarrow \R$ denote unknown functions to the problem (\ref{System_Main}). Moreover, the constant $\varepsilon>0$ presents the size of the initial data and the given functions $(u_0,u_1,v_0,v_1)$ represent the shape of the initial data.

To get started, let us sketch out several historical background related to our models \eqref{Equation_Main} and \eqref{System_Main} in the Euclidean space $\R^n$ by considering the following Cauchy problems for the nonlinear classical damped wave equation:
\begin{align} \label{IntrDampedWaveSemilin}
\begin{cases}
u_{tt}(t,x)-\Delta u(t,x)+u_t(t,x)=|u(t,x)|^p, & (t,x)\in (0,T)\times \R^n, \\
u(0,x)= \e u_0(x), \quad u_t(0,x)= \e u_1(x), & x\in\R^n,
\end{cases}
\end{align}
and
\begin{align}\label{Eq_Coupled_Damped_Waves}
\begin{cases}
u_{tt}(t,x)-\Delta u(t,x)+u_t(t,x)=|v(t,x)|^p, &(t,x)\in (0,T)\times \R^n,\\
v_{tt}(t,x)-\Delta v(t,x)+v_t(t,x)=|u(t,x)|^q, &(t,x)\in (0,T)\times \R^n,\\
u(0,x)= \e u_0(x),\quad u_t(0,x)= \e u_1(x), &\quad x\in \R^n, \\
v(0,x)= \e v_0(x),\quad v_t(0,x)= \e v_1(x), &\quad x\in \R^n,
\end{cases}
\end{align}
where $p,q>1$. First of all, thanks to the presence of the damping term $u_t$, some decay estimates for solutions to the corresponding linear equation of \eqref{IntrDampedWaveSemilin} were derived in the two pioneering papers \cite{Matsumura76,Matsumura77} since 1976 and the asymptotic behavior of solutions to \eqref{IntrDampedWaveSemilin} was also established there simultaneously. Afterwards, we recognize the critical exponent of \eqref{IntrDampedWaveSemilin} is
$$p_{\text{Fuj}}(n):=1+\frac{2}{n}, $$
the so-called Fujita exponent coming from the corresponding semi-linear heat equations (see \cite{Fujita-1966} and references therein). Namely, by assuming compactly supported small data the authors in \cite{TodorovaYordanov2001} proved a result for global (in time) solution existence to \eqref{IntrDampedWaveSemilin}  in the case $p>p_{\text{Fuj}}(n)$ and $p\leq n/(n-2)$ when $n\geq3$. Under some integral sign conditions for the initial data, a blow-up result for local (in time) solutions was demonstrated when $1<p<p_{\text{Fuj}}(n)$. Especially, the author in \cite{Zhang2001} applied the so-called test function method, which was originally developed in \cite{BarasPierre}, to indicate that the critical case $p=p_{\text{Fuj}}(n)$ belongs to the blow-up range. In particular, this method is based on a contradiction argument and finds out sharp exponents for parabolic like models. Among other things, sharp estimates for the maximal existence time when local solutions blow up in finite time, which is also called the lifespan of solutions to \eqref{IntrDampedWaveSemilin}, have been studied in several papers \cite{Kira-Qafs-2002,IkedaOgawa2016,Iked-Waka-2015,Lai-Zhou-2019} as follows:
\begin{align} \label{Lifespan.Eq}
	{\rm LifeSpan}(u)\sim \begin{cases}
		C\varepsilon^{-\frac{2(p-1)}{2-n(p-1)}}&\mbox{if}\ \ 1<p<p_{\mathrm{Fuj}}(n),\\
		\exp\left(C\varepsilon^{-(p-1)}\right)&\mbox{if}\ \ p=p_{\mathrm{Fuj}}(n),
	\end{cases}
\end{align}
with a positive constant $C=C(n,p,u_0,u_1)$ independent of $\varepsilon$. From these observations, it seems to be complete in terms of investigating the problem \eqref{IntrDampedWaveSemilin} in 2019. Secondly, turning to the following Cauchy problem for a weakly coupled system of nonlinear damped wave equations \eqref{Eq_Coupled_Damped_Waves} one realizes that the critical condition \begin{align}\label{Critical_Condition}
\Gamma(p,q):=\frac{\max\{p,q\}+1}{pq-1}=\frac{n}{2},
\end{align}
becomes the critical curve in the $p-q$ plane, which can be described by the following figure:

\begin{figure}[http]
	\centering
	\begin{tikzpicture}
	\fill[domain=2:4.8,color=black!10!white] plot[smooth, tension=.7] coordinates {(1.2,5.3) (1.7,3.6) (2.5,2.5)}--(1.2,5.3) -- (1,5.3) -- (1,1) -- (2.5,2.5)--cycle;
	\fill[domain=2:4.8,color=black!10!white] plot[smooth, tension=.7] coordinates {(2.5,2.5) (3.6,1.7) (5.3,1.2)}--(5.3,1.2) -- (5.3,1) -- (1,1) -- (2.5,2.5)--cycle;
	\fill[domain=2:4.8,color=black!20!white] plot[smooth, tension=.7] coordinates {(1.2,5.3) (1.7,3.6) (2.5,2.5)}--(2.5,2.5) -- (5.3,5.3) -- (1.2,5.3)--cycle;
	\fill[domain=2:4.8,color=black!20!white] plot[smooth, tension=.7] coordinates {(2.5,2.5) (3.6,1.7) (5.3,1.2)}--(5.3,1.2) -- (5.3,5.3) -- (2.5,2.5)--cycle;

	\draw[->] (-0.2,0) -- (5.8,0) node[below] {$p$};
	\draw[->] (0,-0.2) -- (0,5.4) node[left] {$q$};
	\draw[dashed, color=black]  (0, 0)--(5.3,5.3);
	\node[right] at (3.3,3.3) {{$\leftarrow$ $p=q$}};
	\node[left] at (0,-0.2) {{$0$}};
	\draw[fill] (1,0) circle[radius=1pt];
	\draw[fill] (2.5,2.5) circle[radius=1pt];
	\node[below] at (1,0) {{$1$}};
	\draw[fill] (0,1) circle[radius=1pt];
	\node[left] at (0,1) {{$1$}};
	\draw[fill] (2.5,0) circle[radius=1pt];
	\node[below] at (2.5,0) {{$p_{\mathrm{Fuj}}(n)$}};
	\draw[fill] (0,2.5) circle[radius=1pt];
	\node[left] at (0,2.5) {{$p_{\mathrm{Fuj}}(n)$}};
	\node[left] at (4.1,4.5) {{$\leftarrow$ $\Gamma(p,q)=\frac{n}{2}$}};
	\draw[dashed, color=black]  (0, 1)--(5.6, 1);
	\draw[dashed, color=black]  (1, 0)--(1, 5.4);
	\draw[dashed, color=black] (2.5,0)--(2.5,2.5);
	\draw[dashed, color=black] (0,2.5)--(2.5,2.5);
	\draw[color=black] plot[smooth, tension=.7] coordinates {(1.2,5.3) (1.7,3.6) (2.5,2.5)};
	\draw[color=black] plot[smooth, tension=.7] coordinates {(2.5,2.5) (3.6,1.7) (5.3,1.2)};
	\fill[color=black!20!white] (6,4.4)--(6.4,4.4)--(6.4,4)--(6,4)--cycle;
    \node[right] at (6.5,4.2) {{Global (in time) solution}};
    \fill[color=black!10!white] (6,3.8)--(6.4,3.8)--(6.4,3.4)--(6,3.4)--cycle;
	\node[right] at (6.5,3.6) {{Blow-up in finite time}};
	\end{tikzpicture}
	\caption{The critical curve for the system \eqref{Eq_Coupled_Damped_Waves} in the $p-q$ plane}
	\label{imggg}
\end{figure}
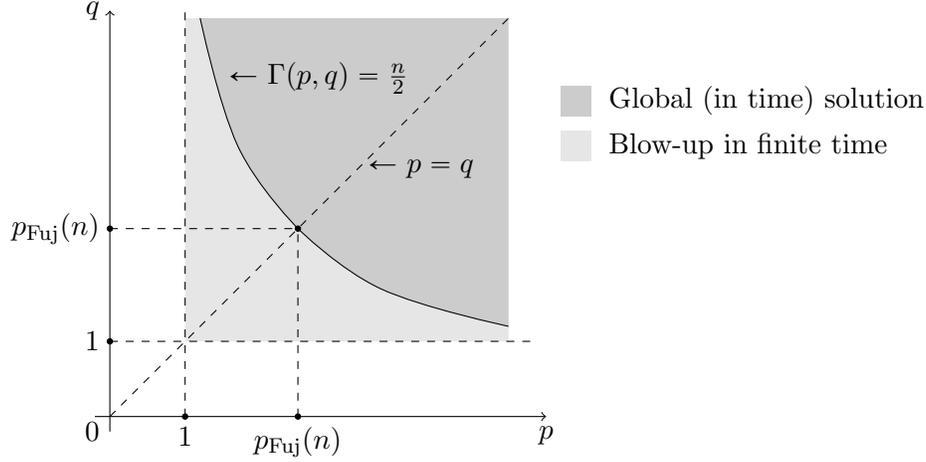
\noindent In other words, a global (in time) small data Sobolev solution exists uniquely when $\Gamma(p,q)<n/2$, meanwhile, every non-trivial local (in time) weak solution, in general, blows up in finite time when $\Gamma(p,q)\geqslant n/2$. Namely, the authors in \cite{Sun-Wang-2007,Nara-2009,Takeda-2009} determined the critical condition \eqref{Critical_Condition} in the low dimensional cases $n=1,2,3$. Later in \cite{Nish-Waka-2014,Nish-Waka-2015}, the critical condition \eqref{Critical_Condition} was claimed for any spatial dimensions $n\geqslant 1$, and moreover, the lifespan of solutions were estimated by
\begin{align*}
	C_1\varepsilon^{-\frac{1}{\Gamma(p,q)-n/2}+\epsilon_0}\le {\rm LifeSpan}(u,v) \le C_2\varepsilon^{-\frac{1}{\Gamma(p,q)-n/2}}
\end{align*}
where $C_k= C_k(n,p,q,u_0,u_1,v_0,v_1)$ with $k=1,2$ are positive constants independent of $\varepsilon$ and $\epsilon_0>0$ a small number, in the subcritical case $\Gamma(p,q)>n/2$ only. However, there is a gap between lifespan bounds from the below to the above to conclude that such estimates in \cite{Nish-Waka-2014,Nish-Waka-2015} are not really optimal. This lack has been filled in the quite recent paper \cite{ChenDao2021} not only in the subcritical case $\Gamma(p,q)>n/2$ but also in the critical case $\Gamma(p,q)=n/2$. More precisely, the authors constructed two test functions with suitable scaling associated with dealing with a system of two nonlinear differential inequalities to derive upper bound estimates for the lifespan in the critical case. On the other hand, they introduced suitable Sobolev spaces with their corresponding norms carrying suitable polynomial-logarithmic type time-dependent weighted functions in terms of obtaining lower bound estimates for the lifespan in both the subcritical and critical cases. Hence, the sharp estimates for the lifespan are established so far as follows:
	\begin{align}\label{Lifespan.Sys}
		{\rm LifeSpan}(u,v) \sim \begin{cases}
			C\varepsilon^{-\frac{1}{\Gamma(p,q)-n/2}} &\text{ if }\, \Gamma(p,q)>n/2, \\
			\mathrm{exp}\left(C\e^{-(p-1)}\right) &\text{ if }\Gamma(p,q)=n/2,\ \ p=q, \\
			\mathrm{exp}\left(C\e^{-\max\left\{\frac{p(pq-1)}{p+1},\frac{q(pq-1)}{q+1}\right\}}\right) &\text{ if }\Gamma(p,q)=n/2,\ \ p\neq q.
		\end{cases}
	\end{align}

Coming back to our interest in this paper we can say that the Cauchy problem for partial differential equations on graphs has achieved a great attention from many authors during the last decades. To be specific, we would like to mention important contributions following this direction in the typical monographs \cite{Grigoryan-2018,KellerLenzWojciechowski-2021,Mugnolo-2016} and several previous works related to elliptic equations/inequalities or their corresponding systems \cite{HuaJost-2014,HuaKeller-2014,MonticelliPunzoSomaglia-2023,GuHuangSun-2023,ImbesiBisci-2023,NguyenDaoDuong-20251,{NguyenDaoDuong-20252}} in which some existence and nonexistence results were established under some assumptions on weighted volume growth of balls and a suitable distance. Also, the setting of parabolic equations on graphs has widely studied in numerous papers (see, for instance, \cite{BarlowCoulhonGrigoryan-2001,HuaMugnolo-2015,Huang-2012,LenzSchmidtZimmermann-2023,Meglioli-2025} and the references therein) to investigate the uniqueness of solutions to a heat equation class. Recently, thanks to heat kernel estimates, the authors in \cite{LinWu-2017,Wu-2018,Wu-2021,LinWu-2018} succeeded in proving both existence and nonexistence of nonnegative global solutions for semilinear heat equations by assuming a curvature dimension condition. Among other things, some estimates and the asymptotic behavior of the lifespan of solutions to a semilinear heat equation are indicated in \cite{HuWang-2024} by developing the first eigenvalue method of Kaplan combined with the
discrete principle of Phragm\'{e}n-Lindel\"{o}f and the approach of upper/lower solutions on graphs. Quite recently, the authors in a series of their research \cite{MonticelliPunzoSomaglia-2024,GrilloMeglioliPunzo-2024,PunzoSacco-2025} have proved the existence of global solutions based on a fixed-point argument with some kind of different nonlinearities. On the other hand, blow-up results for local solutions have been demonstrated in \cite{GrilloMeglioliPunzo-2024,PunzoSacco-2025} by the aid of the contradiction principle, or in \cite{MonticelliPunzoSomaglia-2024} by using estimates for the Laplacian of the distance and conditions for suitable weighted space-time volume growth as well. Then, the authors in \cite{MonticelliPunzoSomaglia-2025} have continued to apply such approaches for the blow-up part to the wave equation on weighted graphs with a nonlinear term of positive potential type. To the best of authors' knowledge, there is no paper in terms of the study of the Cauchy problems for nonlinear damped wave equations on a weighted graph so far. The main motivation is strongly inspired by the recent papers \cite{MonticelliPunzoSomaglia-2024,MonticelliPunzoSomaglia-2025} to prove results for blow-up in finite time of solutions and give upper bound estimates for the lifespan of solutions to \eqref{Equation_Main} and \eqref{System_Main} simultaneously. For the blow-up situation, we are going to utilize the refined test function method combined with the same techniques as in \cite{MonticelliPunzoSomaglia-2024,MonticelliPunzoSomaglia-2025} to achieve point-wise estimates for the Laplacian of a judicious distance. The point worth noticing is that the proof of lifespan estimates for solutions relies on a novel technique which has never appeared in previous literature even not a simple generalization of the method used in \cite{HuWang-2024}. More precisely, constructing refined test functions with the help of some parameter-dependent auxiliary functionals plays an essential role in catching our lifespan estimates. Finally, the optimality of our results will be also discussed in this paper by recovering both the Fujita exponent and the critical curve as in the Euclidean space when we apply them to the $n$-dimensional integer lattice graph $\Z^n$.

\subsection{Notations} 
We give the following notations which are used throughout this paper.
\begin{itemize}[leftmargin=*]
\item We call a weighted graph as the triple $(V,\omega,\mu)$ defined by $\mu: V \to [0,\ity)$ and $\omega: V\times V \to [0,\ity)$ be two given functions fulfilling the following conditions:
\begin{itemize}
    \item[i)] \textit{Symmetric property}: $\omega(x,y) = \omega(y,x)$ for any $(x,y)\in V\times V$,
    \item[ii)] \textit{Zero diagonal}: $\omega(x,x) = 0$ for any $x\in V$,
    \item[iii)] \textit{Finite sum}: $\dps\sum_{y\in V}\omega(x,y) < \ity$ for any $x \in V$.
\end{itemize}
Sometimes, $\mu$ and $\omega$ are the so-called edge weight and node measure, respectively.
\item We say that a weighted graph $(V,\omega,\mu)$ is
\begin{itemize}
    \item[i)] \textit{locally finite} if each vertex $x$ has only finitely many $y\in V$ such that $y\sim x$ (we call that $y$ is connected to $x$, or a couple $(x,y)$ is an edge in $E$), i.e. $\omega(x,y)>0$.
    \item[ii)] \textit{connected} if for any two distinct vertices $x$ and $y$, there exists a path connecting $x$ and $y$, which is understood to be a collection of vertices $\{x_k\}_{k=0}^n$ for some $n\in\N$ such that $x_0=x$, $x_n=y$ and $x_k\sim x_{k+1}$ for any $k=0,1,\cdots,n-1$.
    \item[iii)] \textit{undirected} if the edge connecting $x$ and $y$ is the same as the edge connecting $y$ and $x$ for any $x,y \in V$.
\end{itemize}
\item We define the distance function $d: V\times V \to [0,\ity)$ which is symmetric, zero diagonal and satisfies
$$ d(x,y) \le d(x,z)+ d(z,y) \quad \text{ for any }x,y,z\in V. $$
\item For any real function $f:V \to \R$, the Laplace operator $\Delta f$ is defined by
$$ \Delta f(x)= \f{1}{\mu(x)}\sum_{y\in V,\, y\sim x}\omega(x,y)\big(f(y)-f(x)\big) \quad \text{ for any } x\in V. $$
\item We denote by $\mathcal{C}(V)$, the set of real functions on $V$. Moreover, we introduce the spaces as follows:
\begin{align*}
    W^{1,p}(V) &= \left\{f \in \mathcal{C}(V): \| f\|_{W^{1,p}(V)} := \left(\sum_{x\in V}\mu(x)|\big(\nabla f(x)|^p +|f(x)|^p\big)\right)^{\frac{1}{p}}<\infty\right\}, \\
    L^p(V) &= \left\{f \in \mathcal{C}(V): \| f\|_{L^p(V)} := \left(\sum_{x\in V}\mu(x)|f(x)|^p\right)^{\frac{1}{p}}<\infty\right\}, \\
    L^\ity(V) &= \left\{f \in \mathcal{C}(V): \| f\|_{L^\ity(V)} := \sup_{x\in V} |f(x)| <\infty\right\},
\end{align*}
where $1\le p<\ity$.
\item For $r \ge 0$, $B(x,r):= \{y\in V: d(x,y) \le r\}$ stands for a ball centered at $x$ with radius $r$. Moreover, for any given set $\Omega\subset V$, the volume of $\Omega$ is given by
$$ {\rm Vol}(\Omega) =\sum_{x\in\Omega} \mu(x). $$
\item We write $f\lesssim g$ when there exists a constant $C>0$ such that $f\le Cg$, and $f \sim g$ when $g\lesssim f\lesssim g$. Hereafter, $C$ is a suitable positive constant which may have different value from line to line.
\end{itemize}

\begin{lemma} \label{Inte-by-parts.Lemma}
    By the definition of Laplace operator above, it is clear to verify the following formula of integration by parts (see more \cite{Grigoryan-2018}):
    $$ \sum_{x\in V} \big(\Delta f(x)\big)g(x)\mu(x)= \sum_{x\in V} f(x)\big(\Delta g(x)\big)\mu(x). $$
\end{lemma}
    
\subsection{Main results}
First of all, in order to state our main result let us give the definitions of weak solutions to \eqref{Equation_Main} and \eqref{System_Main}.

\begin{definition}\label{Defn1_Weak}
    The function $u= u(t,x)$ is called a local (in time) weak solution to the Cauchy problem \eqref{Equation_Main} on $[0,T)$ with $T>0$, if $u\in L_{\rm loc}^p\big([0,T)\times V\big)$ satisfies the following integral equality:
\begin{align}
&\int_0^T \sum_{x\in V} \mu(x) u(t,x) \partial_t^2\Psi(t,x)dt - \int_0^T \sum_{x\in V} \mu(x)\Delta u(t,x) \Psi(t,x)dt \notag \\
&\hspace{5cm}- \int_0^T \sum_{x\in V} \mu(x) u(t,x)\partial_t\Psi(t,x)dt \notag\\
&\qquad=\int_0^T \sum_{x\in V} \mu(x)\Psi(t,x)|u(t,x)|^p dt+ \varepsilon \sum_{x\in V} \mu(x) \big((u_0(x)+u_1(x)\big)\Psi(0,x) \notag \\
&\hspace{5cm}- \varepsilon \sum_{x\in V} \mu(x)u_0(x)\partial_t\Psi(0,x), \label{Integral_00}
\end{align}
for any $\Psi\in\mathcal{C}_0^{\infty}\big([0,T)\times V\big)$. If $T= \ity$, we say that the function $u$ is a global (in time) weak solution to the Cauchy problem \eqref{Equation_Main}.
\end{definition}

\begin{definition}\label{Defn2_Weak}
    The pair of functions $(u,v)= (u(t,x),v(t,x))$ is called a local (in time) weak solution to the Cauchy problem \eqref{System_Main} on $[0,T)$ with $T>0$, if 
    $$ (u,v)\in L_{\rm loc}^q\big([0,T)\times V\big)\times L_{\rm loc}^p\big([0,T)\times V\big) $$
    satisfies the following integral equalities:
\begin{align} \label{Integral_01}
&\int_0^T \sum_{x\in V} \mu(x) u(t,x) \partial_t^2\Psi_1(t,x)dt - \int_0^T \sum_{x\in V} \mu(x)\Delta u(t,x) \Psi_1(t,x)dt \notag \\
&\hspace{5cm}- \int_0^T \sum_{x\in V} \mu(x) u(t,x)\partial_t\Psi_1(t,x)dt \notag\\
&\qquad=\int_0^T \sum_{x\in V} \mu(x)\Psi_1(t,x)|v(t,x)|^p dt+ \varepsilon \sum_{x\in V} \mu(x) \big((u_0(x)+u_1(x)\big)\Psi_1(0,x) \notag \\
&\hspace{5cm}- \varepsilon \sum_{x\in V} \mu(x)u_0(x)\partial_t\Psi_1(0,x),
\end{align}
as well as
\begin{align} \label{Integral_02}
&\int_0^T \sum_{x\in V} \mu(x) v(t,x) \partial_t^2\Psi_2(t,x)dt - \int_0^T \sum_{x\in V} \mu(x)\Delta v(t,x) \Psi_2(t,x)dt \notag \\
&\hspace{5cm}- \int_0^T \sum_{x\in V} \mu(x) v(t,x)\partial_t\Psi_2(t,x)dt \notag\\
&\qquad=\int_0^T \sum_{x\in V} \mu(x)\Psi_2(t,x)|u(t,x)|^q dt+ \varepsilon \sum_{x\in V} \mu(x) \big((v_0(x)+v_1(x)\big)\Psi_2(0,x) \notag \\
&\hspace{5cm}- \varepsilon \sum_{x\in V} \mu(x)v_0(x)\partial_t\Psi_2(0,x),
\end{align}
for any $\Psi_1,\Psi_2\in\mathcal{C}_0^{\infty}\big([0,T)\times V\big)$. If $T= \ity$, we say that the pair of functions $(u,v)$ is a global (in time) weak solution to the Cauchy problem \eqref{System_Main}.
\end{definition}

\begin{definition}\label{Defn_Lifespan}
The quantity $T_\varepsilon$, which is defined by
\begin{align*}
T_\varepsilon:= \sup \big\{T\in (0,\ity) : &\mbox{ There exists a unique local (in time) weak solution } \notag \\
	&\,\,\, u \mbox{ to \eqref{Equation_Main}}/(u,v) \mbox{ to \eqref{System_Main} on } [0,T) \mbox{ in the sense of Definition \ref{Defn1_Weak}/\ref{Defn2_Weak}} \notag\\
    &\mbox{ with a fixed parameter }\varepsilon>0\big\},
\end{align*}
is called the lifespan (or the so-called maximum existence time) of solutions to the problem \eqref{Equation_Main}/\eqref{System_Main}, respectively.
\end{definition}

Our main results of this paper read as follows.
\begin{theorem} \label{theorem1.1}
    Let us assume that the graph $V$ is infinite and the initial data $u_0,u_1 \in \mathcal{C}^\ity_0(V)$ satisfy
	\begin{align}\label{Assumption1_Initial_data}
		\sum_{x\in V}\mu(x)\big(u_0(x)+ u_1(x)\big) >0.
	\end{align}
	If $p>1$ fulfills the condition for volume growth as follows:
    \begin{equation} \label{Critical1_Condition}
        {\rm Vol}(B(x_0,R))\lesssim R^{\frac{1+\nu}{p-1}}
    \end{equation}
    for any sufficiently large $R$, where $x_0$ and $\nu$ are defined as in Section \ref{Sec.2.1}, then every non-trivial local (in time) weak solution to \eqref{Equation_Main} blows up in finite time. Moreover, there exists a positive constant $\e_0$ such that for any $\e \in (0,\e_0]$ the lifespan $T_\e$ of weak solutions to the Cauchy problem \eqref{Equation_Main} possesses the following estimates:
    \begin{itemize}[leftmargin=*]
    \item If ${\rm Vol}(B(x_0,R))= o\Big(R^{\frac{1+\nu}{p-1}}\Big)$, then
	\begin{equation}\label{Lifespan11}
		\frac{T_\e}{\Big({\rm Vol}(B(x_0,T_\e^{\frac{1}{1+\nu}}))\Big)^{p-1}} \leqslant C\e^{-(p-1)}.
        \end{equation}
        \item If ${\rm Vol}(B(x_0,R))\sim R^{\frac{1+\nu}{p-1}}$, then
        \begin{equation}\label{Lifespan12}
			T_\e \le \exp\left(C\e^{-(p-1)}\right).
	\end{equation}
    \end{itemize}
	Here $C$ is a positive constant independent of $\e$.
\end{theorem}

\begin{theorem} \label{theorem1.2}
    Let us assume that the graph $V$ is infinite and the initial data $u_j,v_j \in \mathcal{C}^\ity_0(V)$ with $j=0,1$ satisfy
	\begin{align}\label{Assumption2_Initial_data}
		\sum_{x\in V}\mu(x)\big(u_0(x)+ u_1(x)\big) >0 \ \ \text{and}\ \  \sum_{x\in V}\mu(x)\big(v_0(x)+ v_1(x)\big) >0.
	\end{align}
	If $p,q>1$ fulfill the condition for volume growth as follows:
    \begin{equation} \label{Critical2_Condition}
        {\rm Vol}(B(x_0,R))\lesssim R^{(1+\nu)\Gamma(p,q)}
    \end{equation}
    for any sufficiently large $R$, where $x_0$ and $\nu$ are defined as in Section \ref{Sec.2.1}, then every non-trivial local (in time) weak solution to \eqref{System_Main} blows up in finite time. Moreover, there exists a positive constant $\e_0$ such that for any $\e \in (0,\e_0]$ the lifespan $T_\e$ of weak solutions to the Cauchy problem \eqref{System_Main} possesses the following upper bounds:
    \begin{itemize}[leftmargin=*]
    \item If ${\rm Vol}(B(x_0,R))= o\Big(R^{(1+\nu)\Gamma(p,q)}\Big)$, then
	\begin{equation}\label{Lifespan21}
		\frac{T_\e}{\Big({\rm Vol}(B(x_0,T_\e^{\frac{1}{1+\nu}}))\Big)^{\frac{1}{\Gamma(p,q)}}} \leqslant C\e^{-{\frac{1}{\Gamma(p,q)}}}.
        \end{equation}
        \item If ${\rm Vol}(B(x_0,R))\sim R^{(1+\nu)\Gamma(p,q)}$, then
	\begin{align}\label{Lifespan22}
		T_\e \leqslant\begin{cases}
			\exp\left(C\e^{-(p-1)}\right) &\text{if}\ \ p=q, \\
			\exp\left(C\e^{-\frac{\max\{p,q\}}{\Gamma(p,q)}}\right) &\text{if}\ \ p\neq q.
		\end{cases}
	\end{align}
    \end{itemize}
	Here $C$ is a positive constant independent of $\e$.
\end{theorem}

\begin{remark}
    \fontshape{n}
    \selectfont
    Obviously, if $p=q$, then the condition \eqref{Critical2_Condition} is exactly equivalent to \eqref{Critical1_Condition} and the lifespan estimate in \eqref{Lifespan21} really turns to that in \eqref{Lifespan11}. This phenomenon causes from the fact that the system \eqref{System_Main} in this case can be re-expressed as the equation \eqref{Equation_Main}.    
\end{remark}

\begin{example}
    \fontshape{n}
    \selectfont
    If we take $V$ to be the $n$-dimensional integer lattice graph $\Z^n$, then from Example \ref{Example-2.1} we rewrite the conditions \eqref{Critical1_Condition} and \eqref{Critical2_Condition} in the following ways:
    $$ n \le \frac{2}{p-1}, \quad \text{i.e.}\quad p \le 1+\frac{2}{n} = p_{\mathrm{Fuj}}(n) $$
    and
    $$ n \le 2\Gamma(p,q), \quad \text{i.e.}\quad \Gamma(p,q) \ge \frac{n}{2}, $$
    which are the same conditions when the blow-up phenomenon of \eqref{IntrDampedWaveSemilin} and \eqref{Eq_Coupled_Damped_Waves} in finite time occurs in $\R^n$, respectively. Moreover, the achieved estimates \eqref{Lifespan11}-\eqref{Lifespan12} and \eqref{Lifespan21}-\eqref{Lifespan22} for lifespan of solutions to \eqref{Equation_Main} and \eqref{System_Main} are determined as follows:
    \begin{align*}
	   T_\e \lesssim \begin{cases}
		C\varepsilon^{-\frac{2(p-1)}{2-n(p-1)}}&\mbox{if}\ \ 1<p<p_{\mathrm{Fuj}}(n),\\
		\exp\left(C\varepsilon^{-(p-1)}\right)&\mbox{if}\ \ p=p_{\mathrm{Fuj}}(n),
	\end{cases}
    \end{align*}
    and
   	\begin{align*}
		T_\e \lesssim \begin{cases}
			C\varepsilon^{-\frac{1}{\Gamma(p,q)-n/2}} &\text{ if }\, \Gamma(p,q)>n/2, \\
			\exp\left(C\e^{-(p-1)}\right) &\text{ if }\Gamma(p,q)=n/2,\ \ p=q, \\
			\exp\left(C\e^{-\frac{\max\{p,q\}}{\Gamma(p,q)}}\right) &\text{ if }\Gamma(p,q)=n/2,\ \ p\neq q,
		\end{cases}
	\end{align*}
    which coincide with the corresponding upper bounds in \eqref{Lifespan.Eq} and \eqref{Lifespan.Sys}, respectively, again. For these observations, we can say that  the optimality of our results in this paper is partially guaranteed in a comparison with the setting of the Euclidean space. 
\end{example}

\begin{remark}
    \fontshape{n}
    \selectfont
    Although the estimates for lifespan of solutions from the below have not yet reported in this paper, we strongly believe, as analyzed in the above example, that the corresponding lower bounds behave asymptotically as their upper ones to claim the sharpness of our results. For this reason, it is a conjecture arising from this work.  
\end{remark}

\textbf{The organization of this paper is as follows:} In Section \ref{Sec.2}, we give some assumptions on the weighted graph $(V,\omega,\mu)$ and introduce the setting of test functions as well. Then, we present the proofs of Theorems \ref{theorem1.1} and \ref{theorem1.2} in Sections \ref{Sec.3} and \ref{Sec.4}, respectively, including the blow-up result and the estimates for lifespan of solutions simultaneously. Finally, several final remarks will be discussed in Section \ref{Sec.5} in this paper.

\section{Preliminaries} \label{Sec.2}

\subsection{Assumptions} \label{Sec.2.1}
In this section, let us give some assumptions as follows:
\begin{itemize}
    \item[(i)] $(V,\omega,\mu)$ is a connected, locally finite, undirected and weighted graph.
    \item[(ii)] $\dps\sum_{y\sim x} \omega(x,y) \lesssim \mu(x)$ for any $x\in V$.
    \item[(iii)] There exists a distance function $d=d(x,y)$ having a finite jump size $L$, i.e. it holds
    $$ L:= \sup\{d(x,y) : x,y\in V \text{ and }\omega(x,y)>0\}< \ity. $$
    \item[(iv)] $B(x,r)$ is assumed to be a finite set for any $x\in V$.
    \item[(v)] There exists $x_0\in V$, $R_0>1$ and $\nu\in [0,1]$ such that it holds
    $$ |\Delta d(x_0,x)| \lesssim \frac{1}{d(x_0,x)^\nu}\quad \text{ for all }x\in V\setminus B(x_0,R_0). $$
\end{itemize}

\begin{example} \label{Example-2.1}
    \fontshape{n}
    \selectfont
    If we take $V$ to be the $n$-dimensional integer lattice graph $\Z^n$, then we may determine as follows (see, for example, \cite{MonticelliPunzoSomaglia-2024,MonticelliPunzoSomaglia-2025}): $\omega(x,y)=
        \begin{cases}
            1 &\text{ if }x,y \text{ are connected} \\
            0 &\text{ otherwise}
        \end{cases}$, $\mu(x)= 2n$, $d(x,y):= |x-y|$ and $\nu=1$. Moreover, the volume of a ball $B(x,r)$ can be given by
        $$ {\rm Vol}(B(x,r))\sim r^n. $$
\end{example}

\subsection{Setting of test functions} \label{Sec.2.2}
We now introduce a test function $\phi=\phi(r)$ such that
\begin{align*}
\phi\in\mathcal{C}_0^{\infty}([0,\infty))\ \ \mbox{and}\ \ \phi(r):=\begin{cases}
	1&\mbox{if}\ \ r\in[0,1/2],\\
	\mbox{decreasing}&\mbox{if}\ \ r\in(1/2,1),\\
	0&\mbox{if}\ \ r\in[1,\infty).
\end{cases}
\end{align*}
Moreover, another test function $\phi^*=\phi^*(r)$ is also introduced by
\begin{align*}
 \phi^*(r):= \begin{cases}
		0&\mbox{if}\ \ r\in[0,1/2),\\
	\phi(r)&\mbox{if}\ \ r\in[1/2,\infty).
\end{cases}
\end{align*}
Next, for a large parameter $R\in(0,\infty)$, we take $\Phi_R=\Phi_R(t,x)$ and $\Phi_R^*=\Phi_R^*(t,x)$ defined by
\begin{align*}
\Phi_R(t,x):= \left(\phi\left(\frac{t^{\alpha+2} + d(x_0,x)^4}{R^4}\right)\right)^{\beta+2}\ \ \mbox{ and }\ \ \Phi_R^*(t,x):= \left(\phi^*\left(\frac{t^{\alpha+2} +d(x_0,x)^4}{R^4}\right)\right)^{\beta+2},
\end{align*}
respectively, where $\alpha$ and $\beta$ are two nonnegative constants which will be chosen later. It is clear to see that
\begin{align*}
    {\rm supp}\,\Phi_R &\subset \left\{(t,x) \in [0,\ity) \times V \,:\, t^{\alpha+2}+ d(x_0,x)^4 \leqslant R^4\right\}, \\
    {\rm supp}\,\Phi^*_R &\subset \left\{(t,x)\in [0,\ity) \times V \,:\, \frac{R^4}{2} \leqslant t^{\alpha+2}+ d(x_0,x)^4 \leqslant R^4\right\}.
\end{align*}

\begin{lemma} \label{lemma2.2}
For any $(t,x) \in [0,\ity) \times V$, the following auxiliary estimates hold:
\begin{align*}
	\left|\partial_t\Phi_R(t,x)\right| &\lesssim R^{-\frac{4}{\alpha+2}} \big(\Phi_R^*(t,x)\big)^{\frac{\beta+1}{\beta+2}}, \\
	\left|\partial^2_t\Phi_R(t,x)\right| &\lesssim R^{-\frac{8}{\alpha+2}}\big(\Phi_R^*(t,x)\big)^{\frac{\beta}{\beta+2}}, \\
	\left| \Delta \Phi_R(t,x)\right| &\lesssim R^{-(1+\nu)}\big(\Phi_R^*(t,x)\big)^{\frac{\beta+1}{\beta+2}},
\end{align*}
where $\nu$ is defined from Assumption ${\rm (v)}$ in Section \ref{Sec.2.1}.
\end{lemma}

\begin{proof}
We can follow ideas from the proof of Theorem $2.6$ in \cite{MonticelliPunzoSomaglia-2024} (or Theorem $3.1$ in \cite{MonticelliPunzoSomaglia-2025}) with some minor modifications to conclude the desired estimates. However, for the convenience of the reader we will present the proof of this lemma in detail. First, straightforward calculations give
\begin{align*}
    \partial_t\Phi_R(t,x)&= (\alpha+2)(\beta+2) \frac{t^{\alpha+1}}{R^4} \left(\phi\left(\frac{t^{\alpha+2} + d(x_0,x)^4}{R^4}\right)\right)^{\beta+1}\phi'\left(\frac{t^{\alpha+2} + d(x_0,x)^4}{R^4}\right)
\end{align*}
and
\begin{align*}
    \partial_t^2\Phi_R(t,x)&= ({\alpha+2})(\alpha+1)(\beta+2) \frac{t^\alpha}{R^4}  \left(\phi\left(\frac{t^{\alpha+2} + d(x_0,x)^4}{R^4}\right)\right)^{\beta+1}\phi'\left(\frac{t^{\alpha+2} + d(x_0,x)^4}{R^4}\right) \\
    &\quad +(\alpha+2)^2(\beta+1)(\beta+2) \frac{t^{2\alpha+2}}{R^8} \left(\phi\left(\frac{t^{\alpha+2} + d(x_0,x)^4}{R^4}\right)\right)^{\beta} \left(\phi'\left(\frac{t^{\alpha+2} + d(x_0,x)^4}{R^4}\right)\right)^2 \\
    &\quad +(\alpha+2)^2(\beta+2) \frac{t^{2\alpha+2}}{R^8} \left(\phi\left(\frac{t^{\alpha+2} + d(x_0,x)^4}{R^4}\right)\right)^{\beta+1} \phi''\left(\frac{t^{\alpha+2} + d(x_0,x)^4}{R^4}\right).
\end{align*}
As the properties
\begin{align*}
    \phi'\left(\frac{t^{\alpha+2} + d(x_0,x)^4}{R^4}\right) \not\equiv0 \quad \text{ and }\quad \phi''\left(\frac{t^{\alpha+2} + d(x_0,x)^4}{R^4}\right) \not\equiv 0
\end{align*}
for any $(t,x)$ fulfilling
$$ \frac{R^4}{2}< t^{\alpha+2} + d(x_0,x)^4 <R^4 $$
remain valid, one may claim the first two estimates in this lemma by noticing the relation
$$ 0<\Phi_R^*(t,x)< 1. $$
What's more, we shall use Assumptions in Section \ref{Sec.2.1} as a key tool to estimate $\Delta \Phi_R(t,x)$. Namely, we obtain
\begin{align}
    -\Delta \Phi_R(t,x) &= -\f{1}{\mu(x)}\sum_{y\in V,\,y\sim x}\omega(x,y)\big(\Phi_R(t,y)- \Phi_R(t,x)\big) \notag \\
    &\le -(\beta+2)\left(\phi\left(\frac{t^{\alpha+2} + d(x_0,x)^4}{R^4}\right)\right)^{\beta+1} \notag \\
    &\qquad \times \f{1}{\mu(x)}\sum_{y\in V,\,y\sim x} \omega(x,y)\left(\phi\left(\frac{t^{\alpha+2} + d(x_0,y)^4}{R^4}\right)-\phi\left(\frac{t^{\alpha+2} + d(x_0,x)^4}{R^4}\right)\right) \notag \\
    &\quad =: (\beta+2)\left(\phi\left(\frac{t^{\alpha+2} + d(x_0,x)^4}{R^4}\right)\right)^{\beta+1}F(t,x), \label{Est.Lemma-1}
\end{align}
where we have utilized the elementary inequality
$$ b^{\beta+2}- a^{\beta+2} \ge (\beta+2)a^{\beta+1}(b-a) \quad \text{for any } a,b\ge 0. $$
By the aid of Taylor expansion, we can express as follows:
\begin{align}
    &\phi\left(\frac{t^{\alpha+2} + d(x_0,y)^4}{R^4}\right)-\phi\left(\frac{t^{\alpha+2} + d(x_0,x)^4}{R^4}\right) \notag \\
    &\quad = \frac{d(x_0,y)^4-d(x_0,x)^4}{R^4}\phi'\left(\frac{t^{\alpha+2} + d(x_0,x)^4}{R^4}\right) \notag \\
    &\qquad + \frac{1}{2}\left(\frac{d(x_0,y)^4-d(x_0,x)^4}{R^4}\right)^2 \phi''\left(\frac{t^{\alpha+2} + \lambda d(x_0,y)^4+ (1-\lambda) d(x_0,x)^4}{R^4}\right) \notag \\
    &\quad = \frac{4d(x_0,x)^3\big(d(x_0,y)-d(x_0,x)\big)+ 6 \big(d(x_0,y)-d(x_0,x)\big)^2 d_1^2}{R^4}\phi'\left(\frac{t^{\alpha+2} + d(x_0,x)^4}{R^4}\right) \notag \\
    &\qquad +\frac{8\big(d(x_0,y)-d(x_0,x)\big)^2 d_2^6}{R^8}\phi''\left(\frac{t^{\alpha+2} + \lambda d(x_0,y)^4+ (1-\lambda) d(x_0,x)^4}{R^4}\right) \label{Est.Lemma-2}
\end{align}
for some $\lambda\in[0,1]$ and $d_1,d_2$ fulfilling
$$ \min\{d(x_0,x),d(x_0,y)\} \le d_1,d_2 \le \max\{d(x_0,x),d(x_0,y)\}. $$
From \eqref{Est.Lemma-1} and \eqref{Est.Lemma-2} one recognizes that
\begin{align*}
    F(t,x) &= -\frac{4d(x_0,x)^3}{R^4}\phi'\left(\frac{t^{\alpha+2} + d(x_0,x)^4}{R^4}\right)\Delta d(x_0,x) \\
    &\quad - \f{6}{\mu(x)R^4}\phi'\left(\frac{t^{\alpha+2} + d(x_0,x)^4}{R^4}\right)\sum_{y\in V,\,y\sim x} \omega(x,y) \big(d(x_0,y)-d(x_0,x)\big)^2 d_1^2\\
    &\quad -\f{8}{\mu(x)R^8}\sum_{y\in V,\,y\sim x} \omega(x,y) \big(d(x_0,y)-d(x_0,x)\big)^2 d_2^6\phi''\left(\frac{t^{\alpha+2} + \lambda d(x_0,y)^4+ (1-\lambda) d(x_0,x)^4}{R^4}\right) \\
    & =: F_1(t,x)+ F_2(t,x)+ F_3(t,x).
\end{align*}
Applying Assumption {\rm (v)} in Section \ref{Sec.2.1} and the fact that $d(x_0,x)\le R$ on ${\rm supp}\,\Phi^*_R$ we arrive at
$$ F_1(t,x) \lesssim -\frac{d(x_0,x)^{3-\nu}}{R^4}\phi'\left(\frac{t^{\alpha+2} + d(x_0,x)^4}{R^4}\right) \lesssim R^{-(1+\nu)} $$
on ${\rm supp}\,\Phi^*_R$. Thanks to the relations
$$ |d(x_0,y)-d(x_0,x)|\le d(x,y) \le L $$
and
$$ d_1 \le \max\{d(x_0,x),d(x_0,y) \le d(x_0,x)+ L $$
one has
$$ F_2(t,x) \le - \f{6L^2}{\mu(x)R^4}\big(d(x_0,x)+L\big)^2\phi'\left(\frac{t^{\alpha+2} + d(x_0,x)^4}{R^4}\right)\sum_{y\in V,\,y\sim x} \omega(x,y) \lesssim R^{-2} $$
on ${\rm supp}\,\Phi^*_R$, provided that $R$ is sufficiently chosen to guarantee $R>L$. Also, in the previous estimate we have utilized Assumption {\rm (ii)} in Section \ref{Sec.2.1} and $d(x_0,x)\le R$ on ${\rm supp}\,\Phi^*_R$ again. In order to control $F_3(t,x)$, we observe that if
$$ \f{R^4}{2}\le t^{\alpha+2} + \lambda d(x_0,y)^4+ (1-\lambda) d(x_0,x)^4\le R^4, $$
then $(t,x)\in {\rm supp}\,\Phi^*_R$ for a sufficiently large number $R$. By the similar argument as we estimated $F_2(t,x)$, we may conclude
$$ F_3(t,x) \lesssim R^{-2}\left|\phi''\left(\frac{t^{\alpha+2} + \lambda d(x_0,y)^4+ (1-\lambda) d(x_0,x)^4}{R^4}\right)\right| \lesssim R^{-2} $$
on ${\rm supp}\,\Phi^*_R$. Combining all previous estimates one finds
\begin{equation} \label{Est.Lemma-3}
    F_3(t,x) \lesssim R^{-(1+\nu)}
\end{equation}
on ${\rm supp}\,\Phi^*_R$. Therefore, from \eqref{Est.Lemma-1} and \eqref{Est.Lemma-3} we have demonstrated the estimate
$$ -\Delta \Phi_R(t,x) \lesssim R^{-(1+\nu)}\left(\phi\left(\frac{t^{\alpha+2} + d(x_0,x)^4}{R^4}\right)\right)^{\beta+1} $$
on ${\rm supp}\,\Phi^*_R$, which completes our proof.
\end{proof}

\section{Proof of Theorem \ref{theorem1.1}}\label{Sec.3}

\subsection{Blow-up result} \label{Sec.3.1}
Let us assume that $u= u(t,x)$ is a global (in time) weak solution to \eqref{Equation_Main} in the sense of Definition \ref{Defn1_Weak}. At the first stage, let us define the following functional:
\begin{align*}
    P_R &= \int_0^\ity \sum_{x\in V} \mu(x)\Phi_R(t,x)|u(t,x)|^p dt.
\end{align*}
By replacing the test function $\Psi$ in the relation \eqref{Integral_00} by $\Phi_R$ we derive
\begin{align}
&P_R+ \varepsilon \sum_{x\in V} \mu(x) \big((u_0(x)+u_1(x)\big)\Phi_R(0,x) \\
&\qquad =\int_0^\ity  \sum_{x\in V} \mu(x) u(t,x) \partial_t^2 \Phi_R(t,x)dt - \int_0^\ity  \sum_{x\in V} \mu(x)\Delta u(t,x) \Phi_R(t,x)dt \notag \\
&\hspace{5cm}- \int_0^\ity  \sum_{x\in V} \mu(x) u(t,x)\partial_t\Phi_R(t,x)dt \notag\\
&\qquad =: P_{1,R} + P_{2,R} + P_{3,R}, \label{Estimate-00}
\end{align}
where we noticed that $\partial_t\Phi_R(0,x)=0$. Employing Lemma \ref{lemma2.2} and H\"{o}lder's inequality with $\frac{1}{p}+\frac{1}{p'}=1$ we can proceed as follows:
\begin{align}
|P_{1,R}| &\lesssim R^{-\frac{8}{\alpha+2}} \int_0^\ity  \sum_{x\in V} \mu(x) |u(t,x)| \big(\Phi_R^*(t,x)\big)^{\frac{\beta}{\beta+2}} dt \notag \\
&\lesssim R^{-\frac{8}{\alpha+2}}\left( \int_0^\ity  \sum_{x\in V} \mu(x) |u(t,x)|^p \big(\Phi_R^*(t,x)\big)^{\frac{p\beta}{\beta+2}} dt\right)^{\frac{1}{p}} \left( \int_0^\ity \sum_{x\in V, {\rm supp}\,\Phi^*_R} \mu(x) dt\right)^{\frac{1}{p'}} \label{Estimate-01}
\end{align}
and
\begin{align}
|P_{3,R}| &\lesssim R^{-\frac{4}{\alpha+2}} \int_0^\ity  \sum_{x\in V} \mu(x) |u(t,x)| \big(\Phi_R^*(t,x)\big)^{\frac{\beta+1}{\beta+2}} dt \notag \\
&\lesssim R^{-\frac{4}{\alpha+2}}\left( \int_0^\ity  \sum_{x\in V} \mu(x) |u(t,x)|^p \big(\Phi_R^*(t,x)\big)^{\frac{p(\beta+1)}{\beta+2}} dt\right)^{\frac{1}{p}} \left( \int_0^\ity  \sum_{x\in V, {\rm supp}\,\Phi^*_R} \mu(x) dt\right)^{\frac{1}{p'}}. \label{Estimate-02}
\end{align}
Now we use the formula of integration by parts from Lemma \ref{Inte-by-parts.Lemma} to estimate $P_{2,R}$ in the way
$$ P_{2,R}= \int_0^\ity  \sum_{x\in V} \mu(x) u(t,x) \Delta\Phi_R(t,x)dt. $$
Again, the application of Lemma \ref{lemma2.2} and H\"{o}lder's inequality with $\frac{1}{p}+\frac{1}{p'}=1$ leads to
\begin{align}
|P_{2,R}| &\lesssim R^{-(1+\nu)} \int_0^\ity  \sum_{x\in V} \mu(x) |u(t,x)| \big(\Phi_R^*(t,x)\big)^{\frac{\beta+1}{\beta+2}} dt \notag \\
&\lesssim R^{-(1+\nu)}\left( \int_0^\ity  \sum_{x\in V} \mu(x) |u(t,x)|^p \big(\Phi_R^*(t,x)\big)^{\frac{p(\beta+1)}{\beta+2}} dt\right)^{\frac{1}{p}} \left( \int_0^\ity  \sum_{x\in V, {\rm supp}\,\Phi^*_R} \mu(x) dt\right)^{\frac{1}{p'}}. \label{Estimate-03}
\end{align}
Combining the estimates from \eqref{Estimate-00} to \eqref{Estimate-03} one arrives at
\begin{align*}
    &P_R+ \varepsilon \sum_{x\in V} \mu(x) \big((u_0(x)+u_1(x)\big)\Phi_R(0,x) \\
    &\qquad\lesssim R^{-\min\{1+\nu,\frac{4}{\alpha+2}\}}\left( \int_0^\ity  \sum_{x\in V} \mu(x) |u(t,x)|^p \big(\Phi_R^*(t,x)\big)^{\frac{p\beta}{\beta+2}} dt\right)^{\frac{1}{p}} \left( \int_0^\ity  \sum_{x\in V, {\rm supp}\,\Phi^*_R} \mu(x) dt\right)^{\frac{1}{p'}}. 
\end{align*}
Now let us choose
$$ \alpha = \frac{2(1-\nu)}{1+\nu}, \ \ \text{i.e.}\ \ 1+\nu = \frac{4}{\alpha+2} $$
and
$$ \beta\geqslant \frac{2}{p-1}, \ \ \text{i.e.}\ \  \frac{p\beta}{\beta +2}\geqslant 1, $$
which give the following estimate:
\begin{align}
    &P_R+ \varepsilon \sum_{x\in V} \mu(x) \big((u_0(x)+u_1(x)\big)\Phi_R(0,x) \notag \\
    &\qquad\lesssim R^{-(1+\nu)}\left( \int_0^\ity  \sum_{x\in V} \mu(x) |u(t,x)|^p \Phi_R^*(t,x) dt\right)^{\frac{1}{p}} \left( \int_0^\ity  \sum_{x\in V, {\rm supp}\,\Phi^*_R} \mu(x) dt\right)^{\frac{1}{p'}}. \label{Estimate-04}
\end{align}
Due to the assumption \eqref{Assumption1_Initial_data}, there exists a sufficiently large constant $R_1>0$ so that one concludes
\begin{equation*}
\sum_{x\in V} \mu(x) \big((u_0(x)+u_1(x)\big)\Phi_R(0,x) >0
\end{equation*}
for any $R>R_1$. Thus, it follows from \eqref{Estimate-04} that
$$ P_R \lesssim R^{-(1+\nu)}P_R^{\frac{1}{p}} \left( \int_0^\ity  \sum_{x\in V, {\rm supp}\,\Phi^*_R} \mu(x) dt\right)^{\frac{1}{p'}} $$
due to the fact $\Phi_R^*(t,x) \le \Phi_R(t,x)$ for any $(t,x) \in [0,\ity) \times V$. As a consequence, from the last estimate we obtain
\begin{align}
P_R^{1-\frac{1}{p}} &\lesssim R^{-(1+\nu)} \left( \int_0^\ity  \sum_{x\in V, {\rm supp}\,\Phi^*_R} \mu(x) dt\right)^{1-\frac{1}{p}} \label{Estimate-05}.
\end{align} 
Then, the condition \eqref{Critical1_Condition} becomes
$$ {\rm Vol}(B(x_0,R))\lesssim R^{\frac{1+\nu}{p-1}}, $$
which gives
\begin{equation} \label{Estimate-06}
    \int_0^\ity  \sum_{x\in V, {\rm supp}\,\Phi^*_R} \mu(x) dt \lesssim R^{\frac{4}{\alpha+2}}\,{\rm Vol}(B(x_0,R)) \lesssim R^{(1+\nu)\frac{p}{p-1}}.
\end{equation}
Linking \eqref{Estimate-05} and \eqref{Estimate-06} we achieve
$$ P_R^{1-\frac{1}{p}} \lesssim 1,\quad \text{i.e.}\quad P_R \lesssim 1. $$
For this reason, there exists a sufficiently large constant $R_2$ such that
$$ P_R= \int_0^\ity  \sum_{x\in V} \mu(x)\Phi_R(t,x)|u(t,x)|^p dt \le C $$
for any $R>R_2$. Subsequently, it yields immediately
$$ \int_0^\ity  \sum_{x\in V} \mu(x) |u(t,x)|^p \Phi_R^*(t,x) dt \to 0 $$
as $R\to \ity$. Again, the combination of \eqref{Estimate-04} and \eqref{Estimate-06} leads to
\begin{equation}
   P_R+ \varepsilon \sum_{x\in V} \mu(x) \big((u_0(x)+u_1(x)\big)\Phi_R(0,x) \lesssim \left( \int_0^\ity  \sum_{x\in V} \mu(x) |u(t,x)|^p \Phi_R^*(t,x) dt\right)^{\frac{1}{p}}. \label{Estimate-07} 
\end{equation}
Passing $R\to \ity$ in \eqref{Estimate-07} we derive
$$ \sum_{x\in V} \mu(x) \big((u_0(x)+u_1(x)\big) = 0, $$
which is a contradiction to the assumption \eqref{Assumption1_Initial_data}. Hence, our proof is completed.

\subsection{Lifespan estimates} \label{Sec.3.2}
Assume that $u= u(t,x)$ is a local (in time) solution to \eqref{Equation_Main} in $[0,T_\e)\times V$ in the sense of Definition \ref{Defn1_Weak}, where $T_\e$ stands for the lifespan of solutions in the sense of Definition \ref{Defn_Lifespan}. Let us divide the proof of Theorem \ref{theorem1.1} into two cases as follows:
\begin{itemize}[leftmargin=*]
\item[$\bullet$] \textbf{Case 1:} If
$$ {\rm Vol}(B(x_0,R))= o\Big(R^{\frac{1+\nu}{p-1}}\Big), $$
then we repeat some steps in Section \ref{Sec.3.1} to get the following estimate:
$$ P_R+ c\e \le C\, \Big({\rm Vol}(B(x_0,R))\Big)^{\frac{1}{p'}}\, R^{-\frac{1+\nu}{p}}P_R^{\frac{1}{p}}, $$
that is,
\begin{equation} \label{LS_1}
c\e \le C\, \Big({\rm Vol}(B(x_0,R))\Big)^{\frac{1}{p'}}\, R^{-\frac{1+\nu}{p}}P_R^{\frac{1}{p}}- P_R,
\end{equation}
where $c$ is a suitable constant fulfilling
$$ \sum_{x\in V} \mu(x) \big((u_0(x)+u_1(x)\big)\Phi_R(0,x)> c> 0 $$
for any $R> R_1$. After applying the following elementary inequality:
$$ A\,y^\gamma- y \le A^{\frac{1}{1-\gamma}} \quad \text{ for any } A>0,\, y \ge 0 \text{ and } 0< \gamma< 1, $$
to (\ref{LS_1}), one conclude that
$$ \e \le C\,R^{-\frac{1+\nu}{p-1}}{\rm Vol}(B(x_0,R)). $$
From this, letting $R \to T_\e^{\frac{\alpha+2}{4}}$ in the previous estimate we may arrive at
$$ \frac{T_\e}{\Big({\rm Vol}(B(x_0,T_\e^{\frac{1}{1+\nu}}))\Big)^{p-1}} \le C\varepsilon^{-(p-1)}, $$
which is the first desired estimate in \eqref{Lifespan11}.
\item[$\bullet$] \textbf{Case 2:} If
$$ {\rm Vol}(B(x_0,R)) \sim R^{\frac{1+\nu}{p-1}}, $$
then we recall the estimate \eqref{Estimate-07} to achieve
\begin{equation} \label{LS_9}
    P_R+ \varepsilon \sum_{x\in V} \mu(x) \big((u_0(x)+u_1(x)\big)\Phi_R(0,x) \lesssim \left( \int_0^T  \sum_{x\in V} \mu(x) |u(t,x)|^p \Phi_R^*(t,x) dt\right)^{\frac{1}{p}}.
\end{equation}
What's more, let us introduce the following auxiliary functions:
$$ h(r):= \int_0^T  \sum_{x\in V} \mu(x) |u(t,x)|^p \Phi_r^*(t,x) dt\quad \text{ with }r\in (0,\ity) $$
and
$$ \mathcal{H}(R):= \int_0^R h(r)r^{-1}dr.$$
Carrying out the change of variable $s=\dfrac{t^{\alpha+2} + d(x_0,x)^4}{r^4}$ one achieves
\begin{align}
\mathcal{H}(R)&=\int_0^R \left(\int_0^T \sum_{x\in V} \mu(x) |u(t,x)|^p \left(\phi^*\left(\frac{t^{\alpha+2} +d(x_0,x)^4}{r^4}\right)\right)^{\beta+2} dt\right)\,r^{-1}dr \nonumber \\
&= \f{1}{4}\int_0^{T} \sum_{x\in V} \mu(x) |u(t,x)|^p \int_{\frac{t^{\alpha+2} + d(x_0,x)^4}{R^4}}^{\infty}\big(\phi^*(s)\big)^{\beta+2}s^{-1}dsdt \nonumber \\
&\leqslant \f{1}{4}\int_0^{T}\sum_{x\in V} \mu(x) |u(t,x)|^p\bigg(\int_{1/2}^1 \big(\phi^*(s)\big)^{\beta+2}s^{-1}ds\bigg)dt\qquad ({\rm supp}\,\phi^* \subset [1/2,1])\nonumber \\
& = \f{1}{4}\int_0^{T}\sum_{x\in V} \mu(x) |u(t,x)|^p\bigg(\int_{1/2}^1 \big(\phi(s)\big)^{\beta+2}s^{-1}ds\bigg)dt\qquad (\phi^* \equiv \phi \text{ in } [1/2,1]) \nonumber \\
&\le \f{1}{4}\int_0^{T}\sum_{x\in V} \mu(x) |u(t,x)|^p \sup_{r \in (0,R)}\left(\phi\left(\frac{t^{\alpha+2} +d(x_0,x)^4}{r^4}\right)\right)^{\beta+2}\bigg(\int_{1/2}^1 s^{-1}\,ds\bigg)\,dt \nonumber \\
&\le \f{\log 2}{4} \int_0^{T}\sum_{x\in V} \mu(x) |u(t,x)|^p \left(\phi\left(\frac{t^{\alpha+2} +d(x_0,x)^4}{R^4}\right)\right)^{\beta+2}\,dt \qquad (\phi \text{ is decreasing}) \nonumber \\
& = \f{\log 2}{4}P_R. \label{LS_10}
\end{align}
Moreover, it is clear to realize that
\begin{equation}
\mathcal{H}'(R)R= h(R)= \int_0^{T}\sum_{x\in V} \mu(x) |u(t,x)|^p \Phi_R^*(t,x)dt. \label{LS_11}
\end{equation}
Linking all estimates from \eqref{LS_9} to \eqref{LS_11} we have established the following estimate:
$$ c\e+ \f{4\mathcal{H}(R)}{\log 2} \le c\e+ P_R \lesssim \big(\mathcal{H}'(R)R\big)^{\frac{1}{p}}, $$
that is,
$$ R^{-1}\left(c\e+ \f{4\mathcal{H}(R)}{\log 2}\right)^p \le C\,\mathcal{H}'(R). $$
Hence, one deduces
\begin{equation} \label{LS_12}
R^{-1}\big(\mathcal{H}(R)\big)^p \le C\left(\f{\log 2}{4}\right)^p \mathcal{H}'(R),\quad \text{i.e.}\quad R^{-1}\le C\left(\f{\log 2}{4}\right)^p \f{\mathcal{H}'(R)}{\big(\mathcal{H}(R)\big)^p}
\end{equation}
and
\begin{equation} \label{LS_13}
c^p\e^p R^{-1} \le C\,\mathcal{H}'(R).
\end{equation}
By setting $r:=R$ and considering $R\geqslant R_1^2$, we take integration of two sides of \eqref{LS_12} over $[\sqrt{R},R]$ to derive
\begin{equation} \label{LS_14}
\f{1}{2}\log R \le -\f{C(\log 2)^p}{(p-1)4^p}\big(\mathcal{H}(r)\big)^{-(p-1)}\Big|_{r=\sqrt{R}}^{r=R} \le \f{C(\log 2)^p}{(p-1)4^p}\big(\mathcal{H}(\sqrt{R})\big)^{-(p-1)}.
\end{equation}
Putting again $r:=R$ and integrating \eqref{LS_13} over $[R_0,\sqrt{R}]$ we have
$$ c^p\e^p\Big(\log\sqrt{R}- \log R_0\Big) \le C\,\Big(\mathcal{H}(\sqrt{R})-\mathcal{H}(R_0)\Big), $$
which gives
$$ \f{1}{4}c^p\e^p\log R \le C\,\mathcal{H}(\sqrt{R}). $$
In the last step, we replace the previous inequality in \eqref{LS_14} to get the estimate
$$ \f{1}{2}\log R \le \f{C(\log 2)^p}{(p-1)4^p}\left(\f{1}{4C}c^p\e^p\log R\right)^{-(p-1)}, $$
which is equivalent to
$$ \log R \le \f{C\log 2}{c^{p-1}}\left(\f{1}{2(p-1)}\right)^{\frac{1}{p}}\e^{-(p-1)}. $$
Therefore, passing $R \to T_\e^{\frac{\alpha+2}{4}}$ in the last estimate we may arrive at
$$ T_\e \le \exp\left(\f{(1+\nu)C\log 2}{c^{p-1}}\left(\f{1}{2(p-1)}\right)^{\frac{1}{p}}\e^{-(p-1)}\right), $$
which is the estimate in \eqref{Lifespan12}.
\end{itemize}
Summarizing, the proof of lifespan estimates in Theorem \ref{theorem1.1} is demonstrated.

\section{Proof of Theorem \ref{theorem1.2}} \label{Sec.4}

\subsection{Blow-up result} \label{Sec.4.1}
Suppose that $(u,v)= (u(t,x),v(t,x))$ is a global (in time) weak solution to \eqref{System_Main} in the sense of Definition \ref{Defn2_Weak}. First of all, let us define the following two functionals:
\begin{align*}
    I_R &= \int_0^\ity \sum_{x\in V} \mu(x)\Phi_R(t,x)|v(t,x)|^p dt, \\
    J_R &= \int_0^\ity \sum_{x\in V} \mu(x)\Phi_R(t,x)|u(t,x)|^q dt.
\end{align*}
Plugging the test function $\Psi_1$ in the relation \eqref{Integral_01} by $\Phi_R$ one obtains
\begin{align}
&I_R+ \varepsilon \sum_{x\in V} \mu(x) \big((u_0(x)+u_1(x)\big)\Phi_R(0,x) \\
&\qquad =\int_0^\ity \sum_{x\in V} \mu(x) u(t,x) \partial_t^2 \Phi_R(t,x)dt - \int_0^\ity \sum_{x\in V} \mu(x)\Delta u(t,x) \Phi_R(t,x)dt \notag \\
&\hspace{5cm}- \int_0^\ity \sum_{x\in V} \mu(x) u(t,x)\partial_t\Phi_R(t,x)dt \notag\\
&\qquad =: J_{1,R} + J_{2,R} + J_{3,R}, \label{Estimate-0}
\end{align}
thanks to the relation $\partial_t\Phi_R(0,x)=0$. The application of Lemma \ref{lemma2.2} and H\"{o}lder's inequality with $\frac{1}{q}+\frac{1}{q'}=1$ gives
\begin{align}
|J_{1,R}| &\lesssim R^{-\frac{8}{\alpha+2}} \int_0^\ity \sum_{x\in V} \mu(x) |u(t,x)| \big(\Phi_R^*(t,x)\big)^{\frac{\beta}{\beta+2}} dt \notag \\
&\lesssim R^{-\frac{8}{\alpha+2}}\left( \int_0^\ity \sum_{x\in V} \mu(x) |u(t,x)|^q \big(\Phi_R^*(t,x)\big)^{\frac{q\beta}{\beta+2}} dt\right)^{\frac{1}{q}} \left( \int_0^\ity \sum_{x\in V, {\rm supp}\,\Phi^*_R} \mu(x) dt\right)^{\frac{1}{q'}} \label{Estimate-1}
\end{align}
and
\begin{align}
|J_{3,R}| &\lesssim R^{-\frac{4}{\alpha+2}} \int_0^\ity \sum_{x\in V} \mu(x) |u(t,x)| \big(\Phi_R^*(t,x)\big)^{\frac{\beta+1}{\beta+2}} dt \notag \\
&\lesssim R^{-\frac{4}{\alpha+2}}\left( \int_0^\ity \sum_{x\in V} \mu(x) |u(t,x)|^q \big(\Phi_R^*(t,x)\big)^{\frac{q(\beta+1)}{\beta+2}} dt\right)^{\frac{1}{q}} \left( \int_0^\ity \sum_{x\in V, {\rm supp}\,\Phi^*_R} \mu(x) dt\right)^{\frac{1}{q'}}. \label{Estimate-2}
\end{align}
By the aid of the formula of integration by parts from Lemma \ref{Inte-by-parts.Lemma} we can proceed $J_{2,R}$ as follows:
$$ J_{2,R}= \int_0^T \sum_{x\in V} \mu(x) u(t,x) \Delta\Phi_R(t,x)dt. $$
After employing Lemma \ref{lemma2.2} and H\"{o}lder's inequality with $\frac{1}{q}+\frac{1}{q'}=1$ again, one has
\begin{align}
|J_{2,R}| &\lesssim R^{-(1+\nu)} \int_0^\ity \sum_{x\in V} \mu(x) |u(t,x)| \big(\Phi_R^*(t,x)\big)^{\frac{\beta+1}{\beta+2}} dt \notag \\
&\lesssim R^{-(1+\nu)}\left( \int_0^\ity \sum_{x\in V} \mu(x) |u(t,x)|^q \big(\Phi_R^*(t,x)\big)^{\frac{q(\beta+1)}{\beta+2}} dt\right)^{\frac{1}{q}} \left( \int_0^\ity \sum_{x\in V, {\rm supp}\,\Phi^*_R} \mu(x) dt\right)^{\frac{1}{q'}}. \label{Estimate-3}
\end{align}
For this reason, we combine the estimates from \eqref{Estimate-0} to \eqref{Estimate-3} to achieve
\begin{align*}
    &I_R+ \varepsilon \sum_{x\in V} \mu(x) \big((u_0(x)+u_1(x)\big)\Phi_R(0,x) \\
    &\qquad\lesssim R^{-\min\{1+\nu,\frac{4}{\alpha+2}\}}\left( \int_0^\ity \sum_{x\in V} \mu(x) |u(t,x)|^q \big(\Phi_R^*(t,x)\big)^{\frac{q\beta}{\beta+2}} dt\right)^{\frac{1}{q}} \left( \int_0^\ity \sum_{x\in V, {\rm supp}\,\Phi^*_R} \mu(x) dt\right)^{\frac{1}{q'}}. 
\end{align*}
Repeating the same procedure as we did above, from the relation \eqref{Integral_02} we also derive
\begin{align*}
    &J_R+ \varepsilon \sum_{x\in V} \mu(x) \big((v_0(x)+v_1(x)\big)\Phi_R(0,x) \\
    &\qquad\lesssim R^{-\min\{1+\nu,\frac{4}{\alpha+2}\}}\left( \int_0^\ity \sum_{x\in V} \mu(x) |v(t,x)|^p \big(\Phi_R^*(t,x)\big)^{\frac{p\beta}{\beta+2}} dt\right)^{\frac{1}{p}} \left( \int_0^\ity \sum_{x\in V, {\rm supp}\,\Phi^*_R} \mu(x) dt\right)^{\frac{1}{p'}}. 
\end{align*}
Now let us choose
$$ \alpha = \frac{2(1-\nu)}{1+\nu}, \ \ \text{i.e.}\ \ 1+\nu = \frac{4}{\alpha+2} $$
and
$$ \beta\geqslant \max \left\{ \frac{2}{p-1},\frac{2}{q-1} \right\}=\frac{2}{\min\{p,q\}-1}, \ \ \text{i.e.}\ \  \frac{\min\{p,q\}\,\beta}{\beta +2}\geqslant 1, $$
which lead to the following estimates:
\begin{align}
    &I_R+ \varepsilon \sum_{x\in V} \mu(x) \big((u_0(x)+u_1(x)\big)\Phi_R(0,x) \notag \\
    &\qquad\lesssim R^{-(1+\nu)}\left( \int_0^\ity \sum_{x\in V} \mu(x) |u(t,x)|^q \Phi_R^*(t,x) dt\right)^{\frac{1}{q}} \left( \int_0^\ity \sum_{x\in V, {\rm supp}\,\Phi^*_R} \mu(x) dt\right)^{\frac{1}{q'}} \label{Estimate-4}
\end{align}
and
\begin{align}
    &J_R+ \varepsilon \sum_{x\in V} \mu(x) \big((v_0(x)+v_1(x)\big)\Phi_R(0,x) \notag \\
    &\qquad\lesssim R^{-(1+\nu)}\left( \int_0^\ity \sum_{x\in V} \mu(x) |v(t,x)|^p \Phi_R^*(t,x) dt\right)^{\frac{1}{p}} \left( \int_0^\ity \sum_{x\in V, {\rm supp}\,\Phi^*_R} \mu(x) dt\right)^{\frac{1}{p'}} \label{Estimate-5}
\end{align}
Because of the assumption \eqref{Assumption2_Initial_data}, there exists a sufficiently large constant $R_3>0$ such that they hold
\begin{equation} \label{Estimate-6}
\sum_{x\in V} \mu(x) \big((u_0(x)+u_1(x)\big)\Phi_R(0,x) >0\quad \text{and}\quad \sum_{x\in V} \mu(x) \big((v_0(x)+v_1(x)\big)\Phi_R(0,x) >0
\end{equation}
for any $R>R_3$. Consequently, one finds from \eqref{Estimate-4} to \eqref{Estimate-6} that
\begin{align*}
    I_R \lesssim R^{-(1+\nu)}J_R^{\frac{1}{q}} \left( \int_0^\ity \sum_{x\in V, {\rm supp}\,\Phi^*_R} \mu(x) dt\right)^{\frac{1}{q'}}, \\
    J_R \lesssim R^{-(1+\nu)}I_R^{\frac{1}{p}} \left( \int_0^\ity \sum_{x\in V, {\rm supp}\,\Phi^*_R} \mu(x) dt\right)^{\frac{1}{p'}},
\end{align*}
since $\Phi_R^*(t,x) \le \Phi_R(t,x)$ for any $(t,x) \in [0,\ity) \times V$. For this reason, it follows immediately
\begin{align}
I_R^{1-\frac{1}{pq}} &\lesssim R^{-(1+\nu)(1+\frac{1}{q})} \left( \int_0^\ity \sum_{x\in V, {\rm supp}\,\Phi^*_R} \mu(x) dt\right)^{1-\frac{1}{pq}}, \nonumber \\ 
J_R^{1-\frac{1}{pq}} &\lesssim R^{-(1+\nu)(1+\frac{1}{p})} \left( \int_0^\ity \sum_{x\in V, {\rm supp}\,\Phi^*_R} \mu(x) dt\right)^{1-\frac{1}{pq}} \label{Estimate-7}.
\end{align} 
Without loss of generality let us assume $p\le q$. Then, the assumption \eqref{Critical2_Condition} becomes
$$ {\rm Vol}(B(x_0,R))\lesssim R^{(1+\nu)\frac{1+q}{pq-1}}, $$
which yields
\begin{equation} \label{Estimate-8}
    \int_0^\ity \sum_{x\in V, {\rm supp}\,\Phi^*_R} \mu(x) dt \lesssim R^{\frac{4}{\alpha+2}}\,{\rm Vol}(B(x_0,R)) \lesssim R^{(1+\nu)\frac{q+pq}{pq-1}}.
\end{equation}
Combining \eqref{Estimate-7} and \eqref{Estimate-8} one gets
$$ J_R^{1-\frac{1}{pq}} \lesssim 1,\quad \text{i.e.}\quad J_R \lesssim 1. $$
This means that there exists a sufficiently large constant $R_4$ such that it holds
$$ J_R= \int_0^\ity \sum_{x\in V} \mu(x)\Phi_R(t,x)|u(t,x)|^q dt \le C $$
for any $R>R_4$. As a result, we may arrive at
$$ \int_0^\ity \sum_{x\in V} \mu(x) |u(t,x)|^q \Phi_R^*(t,x) dt \to 0 $$
as $R\to \ity$. After linking \eqref{Estimate-4}, \eqref{Estimate-5} and \eqref{Estimate-8} we derive
$$ J_R+ \varepsilon \sum_{x\in V} \mu(x) \big((v_0(x)+v_1(x)\big)\Phi_R(0,x) \lesssim \left( \int_0^\ity \sum_{x\in V} \mu(x) |u(t,x)|^q \Phi_R^*(t,x) dt\right)^{\frac{1}{pq}} $$
From the last two estimates, we let $R\to \ity$ to conclude
$$ \sum_{x\in V} \mu(x) \big((v_0(x)+v_1(x)\big) = 0, $$
which is a contradiction to the assumption \eqref{Assumption2_Initial_data}. All in all, the proof of Theorem \ref{theorem1.1} is complete.

\subsection{Lifespan estimates} \label{Sec.4.2}
Let us suppose that $(u,v)= (u(t,x),v(t,x))$ is a local (in time) solution to \eqref{System_Main} in $[0,T_\e)\times V$ in the sense of Definition \ref{Defn2_Weak}, where $T_\e$ stands for the lifespan of solutions in the sense of Definition \ref{Defn_Lifespan}. Now, to prove Theorem \ref{theorem1.2} let us separate our consideration into two cases as follows:
\begin{itemize}[leftmargin=*]
\item[$\bullet$] \textbf{Case 1:} If
$$ {\rm Vol}(B(x_0,R))= o\Big(R^{(1+\nu)\Gamma(p,q)}\Big), $$
then repeating some steps in Section \ref{Sec.4.1} we obtain the following estimate:
$$ J_R+ c\e \le C\, \Big({\rm Vol}(B(x_0,R))\Big)^{1-\frac{1}{pq}}\, R^{-(1+\nu)\frac{1+q}{pq}}J_R^{\frac{1}{pq}}, $$
which is equivalent to
\begin{equation} \label{LS_1}
c\e \le C\, \Big({\rm Vol}(B(x_0,R))\Big)^{1-\frac{1}{pq}}\, R^{-(1+\nu)\frac{1+q}{pq}}J_R^{\frac{1}{pq}}- J_R,
\end{equation}
where $c$ is a suitable constant satisfying
$$ \sum_{x\in V} \mu(x) \big((u_0(x)+u_1(x)\big)\Phi_R(0,x)> c> 0 $$
for any $R> R_3$. Using the elementary inequality
$$ A\,y^\gamma- y \le A^{\frac{1}{1-\gamma}} \quad \text{ for any } A>0,\, y \ge 0 \text{ and } 0< \gamma< 1, $$
to (\ref{LS_1}) we derive
$$ \e \le C\,R^{-(1+\nu)\frac{1+q}{pq-1}}{\rm Vol}(B(x_0,R)). $$
As a consequence, we let $R \to T_\e^{\frac{\alpha+2}{4}}$ in the previous estimate to achieve
$$ \frac{T_\e}{\Big({\rm Vol}(B(x_0,T_\e^{\frac{1}{1+\nu}}))\Big)^{\frac{1}{\Gamma(p,q)}}} \leqslant C\e^{-{\frac{1}{\Gamma(p,q)}}}, $$
which is the first desired estimate in \eqref{Lifespan21}.
\item[$\bullet$] \textbf{Case 2:} If
$$ {\rm Vol}(B(x_0,R))\sim R^{(1+\nu)\Gamma(p,q)}, $$
then recalling the estimates to achieve \eqref{Estimate-4} and \eqref{Estimate-5} one finds
\begin{align*}
    &I_R+ \varepsilon \sum_{x\in V} \mu(x) \big((u_0(x)+u_1(x)\big)\Phi_R(0,x) \\
    &\qquad \lesssim R^{(1+\nu)\frac{(q-1)\Gamma(p,q)-1}{q}}\left( \int_0^\ity \sum_{x\in V} \mu(x) |u(t,x)|^q \Phi_R^*(t,x) dt\right)^{\frac{1}{q}}
\end{align*}
and
\begin{align*}
    &J_R+ \varepsilon \sum_{x\in V} \mu(x) \big((v_0(x)+v_1(x)\big)\Phi_R(0,x) \\
    &\qquad \lesssim R^{(1+\nu)\frac{(p-1)\Gamma(p,q)-1}{p}}\left( \int_0^\ity \sum_{x\in V} \mu(x) |v(t,x)|^p \Phi_R^*(t,x) dt\right)^{\frac{1}{p}}.
\end{align*}
For $r\in (0,\ity)$, let us introduce the following auxiliary functions:
\begin{align*}
    g_p(r) &:= \int_0^T  \sum_{x\in V} \mu(x) |u(t,x)|^p \Phi_r^*(t,x) dt \quad \text{ and }\mathcal{G}_p(R) = \int_0^R g_p(r)r^{-1}dr, \\
     g_q(r) &:= \int_0^T  \sum_{x\in V} \mu(x) |v(t,x)|^q \Phi_r^*(t,x) dt\quad \text{ and }\mathcal{G}_q(R) = \int_0^R g_q(r)r^{-1}dr.
\end{align*}
Then, repeating some arguments to derive \eqref{LS_10} we may arrive at some estimates as follows:
$$ \mathcal{G}_p(R) \le \f{\log 2}{4} I_R \quad \text{ and }\quad \mathcal{G}_q(R) \le \f{\log 2}{4} J_R. $$
Combining these estimates with the relations
$$ g_p(R) = R\mathcal{G}'_p(R)\quad \text{ and }\quad g_q(R) = R\mathcal{G}'_q(R),$$
one obtains
\begin{align*}
    c\e+ \f{4G_q(R)}{\log 2} \le c\e+ J_R &\lesssim R^{(1+\nu)\frac{(p-1)\Gamma(p,q)-1}{p}}\big(R\mathcal{G}'_p(R)\big)^{\frac{1}{p}}, \\
    c\e+ \f{4G_p(R)}{\log 2} \le c\e+ I_R &\lesssim R^{(1+\nu)\frac{(q-1)\Gamma(p,q)-1}{q}}\big(R\mathcal{G}'_q(R)\big)^{\frac{1}{q}},
\end{align*}
which give
\begin{align*}
    \mathcal{G}'_p(R) &\ge C_1 R^{\nu-(1+\nu)(p-1)\Gamma(p,q)}\big(c\e + G_q(R)\big)^p, \\
    \mathcal{G}'_q(R) &\ge C_1 R^{\nu-(1+\nu)(q-1)\Gamma(p,q)}\big(c\e + G_p(R)\big)^q.
\end{align*}
What's more, to control the above system of nonlinear differential inequalities, motivated by the recent paper \cite{ChenDao2021} we can follow the same approach used in this paper to conclude all estimates for lifespan of solutions in \eqref{Lifespan21}.
\end{itemize}
Summarizing, the proof of Theorem \ref{theorem1.2} in complete.

\section{Concluding remarks}\label{Sec.5}

\begin{remark}[\textbf{Wave equations with double damping}]
\fontshape{n}
\selectfont
Throughout in this paper, we have succeeded in obtaining the upper bounds of lifespan estimates for the Cauchy problems \eqref{Equation_Main} and \eqref{System_Main} for nonlinear damped wave equations on a weighted graph $G=(V,E)$. We want to underline that the approach used can be also applicable to catch some estimates for lifespan of solutions to the Cauchy problem for nonlinear wave equations with double damping as follows:
\begin{equation*}
\begin{cases}
u_{tt}(t,x)- \Delta u(t,x)+ u_t(t,x)- \Delta u_t(t,x)= |u(t,x)|^p, &(t,x) \in (0,T) \times V, \\
u(0,x)= \e u_0(x),\quad u_t(0,x)= \e u_1(x), &\quad x\in V,
\end{cases}
\end{equation*}
and its weakly couple system, or even for nonlinear wave equations with parabolic-like damping. This expectation is reasonable since our estimates in this paper exactly coincide with those from \cite{HuWang-2024} in which the authors have dealt with a semi-linear parabolic equation.
\end{remark}

\begin{remark}[\textbf{An open problem}]
\fontshape{n}
\selectfont
An interesting question, which should be proposed here, to give a positive answer is whether one may prove a result for global (in time) existence of solutions to \eqref{Equation_Main} and \eqref{System_Main} or not. The main difficulty comes from the lack of estimates for solutions to the corresponding linear equations as they appear in the setting of Euclidean space. Furthermore, the technique of estimating heat kernels (see, for example, \cite{LinWu-2017,Wu-2018,Wu-2021,LinWu-2018}) does not work well in the situation of damped wave equations. For this reason, we can say that it remains an open problem. 
\end{remark}

\section*{Acknowledgments}
This work of Tuan Anh Dao is supported by Vietnam Ministry of Education and Training and Vietnam Institute for Advanced Study in Mathematics under grant number B2026-CTT-04.


\end{document}